\documentclass{amsart}
\usepackage[T1]{fontenc}
\usepackage{url}
\newtheorem{theorem}[subsection]{Theorem}

\title{Existence and maximal corank of simple $Z_p$-invariant germs}
\author{Ivan Proskurnin}
\date{}

\begin{document}
\maketitle

\begin{abstract}
In this paper we improve the previously achieved upper bound on the corank of an equivariantly stable singularity for a group of prime order.
We also prove that the maximal corank of a simple $\mathbb{Z}_p$-invariant germ tends to infinity as $p$ increases and is asymptotically logarithmic, so the previously obtained bound is valid up to order of magnitude.
\end{abstract}

\section{Introduction}

The present paper is a continuation of our former work. In a previous paper (\cite{1}), a following theorem has been proven:
\medskip

\begin{theorem} Let  $\tau$ be a linear action of $\mathbb{Z}_p$ on $(\mathbb{C}^n,0)$, $rk(\tau)$ the maximal rank of  a $\tau$-invariant quadratic form, $p$ -- a prime number. Germs equivariantly simple with respect to $\tau$ may only exist in one of the two cases:

1)$det(\tau) \neq 1, n - rk(\tau) \leq log_2 (p+1)$;

2)$det(\tau) =1, n - rk(\tau) \leq log_2 (2p-1)$.

\end{theorem}

The statement of theorem 1.1. is not very encouraging, as one may hope that there is a reasonably tame classification of all $\mathbb{Z}_p$-invariant simple singularities. For such classification to exist, there should at least exist some constant upper bound for the coranks of simple germs. Here we will prove that this is in fact not the case and the estimate $n - rk(\tau) = O(log_2(p))$ is valid up to order of magnitude. We will also narrow down the possible range of group actions for which simple singularities may exist. Namely, we are going to prove the following theorems:

\begin{theorem} Let  $\tau$ be a linear action of $\mathbb{Z}_p$ on $(\mathbb{C}^n,0)$, $rk(\tau)$ the maximal rank of  a $\tau$-invariant quadratic form, $p$ -- a prime number. Germs equivariantly simple with respect to $\tau$ may only exist if $det(\tau) \neq 1, n - rk(\tau) \leq log_2 (p+1)$ or if $\tau$ is a real group action.

\end{theorem}

\begin{theorem} 

For each $N \: \in \; \mathbb{N}$ there is a prime number $p$ and a $\mathbb{Z}_p$-invariant simple (with respect to a nontrivial $\mathbb{Z}_p$-action) germ $f: (\mathbb{C}^n, 0) \longrightarrow  (\mathbb{C}, 0)$ with $j^2(f)=0$ and $n>N$. In fact, for every $\varepsilon>0$ for allmost all prime $p$ there is a $\mathbb{Z}_p$-invariant simple germ $f: (\mathbb{C}^n, 0) \longrightarrow  (\mathbb{C}, 0)$ with $n>(1-\varepsilon)(log(log(p)))$ and $j^2(f)=0$.

\end{theorem}

\begin{theorem}There are constants $\alpha, \beta$ and an infinite sequence of distinct prime numbers $p_1, \ldots p_n$ such that there is a simple $\mathbb{Z}_{p_i}$-invariant germ with corank at least $\alpha log_2 (p_i - 1) + \beta$.

\end{theorem}

The last theorem means that the upper bound in the theorem 1.2. is in fact valid up to an order of magnitude: the corank of a simple singularity does in fact grow logarithmically as $p$ increases. 

Since any analytic action of a finite group is linearizable, the assumption that the group action is linear is not limiting: we will work with linear actions, however all the results can be generalized to analytic group actions in an obvious fashion.

\section{Definitions, notation  and auxiliary statements}
The letter $p$ will denote a prime number.
In a suitable local system of coordinates any analytic finite group action on $(\mathbb{C}^n,0)$ can be linearized (see \cite{2}). So we will consider all group actions to be induced by a linear representation of $G$.
Furthermore, any linear representation of a finite abelian group is diagonalazable, hence further we will usually consider only actions of the form $g \circ (x_1,\ldots ,x_n) \longrightarrow (\chi_1(g) x_1,\ldots , \chi_n(g) x_n)$ with $\chi_i$ being group characters of $\mathbb{Z}_p$.

For a function germ $f$ $J_f$ will denote the \textbf{Jacobian ideal} $<\frac{\partial f}{\partial x_1}, \ldots, \frac{\partial f}{\partial x_n}>$.
$Q_f$ is the Jacobian algebra $\mathbb{C}\{x_1, \ldots, x_n\} / J_f$. $Q_f$ always admits a monomial basis, i.e. the basis consisting of monomials.

If $f$ is invariant with respect to an action of a group $G$, the Jacobian ideal is invariant and therefore $G$ also acts on $Q_f$ by changes of coordinates: $g \circ h(x) = h ( g^{-1} (x))$. The representation of $G$ on $Q_f$ is called \textbf{the equvariant Milnor number} of $f$ and denoted $\mu_G(f)$. $\nu(f)$ is the multiplicity of the trivial representation in $\mu_G(f)$, i.e. the dimension of the space of invariant elements of $Q_f$.

If $f=x_1^2 + \ldots + x_k^2 + h(x_{k+1}, \ldots, x_n)$ $Q_f$ if isomorphic to $Q_h$. If $f$ is invariant with respect to some group action, they are isomorphic as group representations. In particular, $\mu(f) = \mu(h)$, $\nu(f) = \nu(h)$.

A group action on $(\mathbb{C}^n, 0)$  is called \textbf{real} if it admits an invariant quadratic form of rank $n$.

\section{Equivariantly simple and equivariantly stable germs}

There are multiple papers on the classification of equivariantly simple singularities (see., e.g., \cite{3}, \cite{4}, \cite{5}). Roughly speaking, an \textbf{equivariantly stable} singularity is the simplest of the equivariantly simple singularities, i.e. the one that is only adjacent to itself. If the case of real group action the equivariantly stable singularity is the nondegenerate one, as the deformation of a Morse function is again a Morse function. For the formal definitions of equivariant simplicity and stability see below.

We will denote $\mathfrak{m}_G$ the maximal ideal in the ring of $G$-invariant analytic germs. Here we only assume G to be finite, not nessesarily abelian or isomorphic to $\mathbb{Z}_p$. Consider $\mathcal{O}^{r}_{G}$ --- the space of $r$-jets of elements in $\mathfrak{m}_G$ with a critical point at the origin.  The group $D^r_G$ of $r$-jets of  $G$-equivariant diffeomorphisms  $(\mathbb{C}^n,0) \longrightarrow (\mathbb{C}^n,0)$ acts on this space.

A $G$-invariant germ $f:(\mathbb{C}^n,0) \longrightarrow (\mathbb{C},0)$ with a critical point at the origin is \textbf{equivariantly simple} if there is an $N \; \in \; \mathbb{N}$ such that for all $r$ large enough  $r$-jet of $f$ has a neighbourhood in $\mathcal{O}^{r}_{G}$ that only intersects at most $N$ $D^r_G$-orbits.

A $G$-invariant germ $f:(\mathbb{C}^n,0) \longrightarrow (\mathbb{C},0)$ with a critical point at the origin is \textbf{equivariantly stable} if for all large enough $r$  $D^r_G$-orbit of $r$-jet of $f$ in  $\mathcal{O}^{r}_{G}$ is open.

\begin{theorem}
$f$ is equivariantly stable iff $\nu(f)=1$.

\end{theorem}
\begin{proof}
Let $f$ be equivariantly stable. As the orbit of $r$-jet of $f$ is open for large enough $r$ the tangent space to the orbit has dimension equal to the dimension of $\mathcal{O}^{r}_{G}$ itself. The tangent space to the orbit  at $r$-jet of $f$ is the projection of the jacobian ideal of $f$ to  $\mathcal{O}^{r}_{G}$. But if for all $r$ large enough the projection of $J_f$ covers all  $\mathcal{O}^{r}_{G}$ $J_f$ contains $\mathfrak{m}_G$. Since $f$ has a critical point at the origin  $J_f$ does not contain 1 and $\nu(f)=1$.

Now assume $\nu(f)=1$. The Milnor number of $f$ is finite, as the invariants of a finite group action have the separating property, i.e. for any two orbits of the $G$-action there is an invariant function that is equal to 1 on the first orbit and 0 on the second. But if $\nu(f)=1$, then $J_f = \mathfrak{m}_G$ and any invariant function takes the same value at any point of the singular locus of $f$ as it does at the origin. So the separating property fails if the critical point of $f$ is non-isolated. Since  $f$ has a finite Milnor number we can apply the Slodowy theorem (see \cite{7}), according to which the invariant versal deformation of $f$ is equal to $f + \lambda_1 h_1 + \ldots + \lambda_{\nu} h_{\nu}$, $h_j$ --- linearly independent invariant elements of $Q_f$. Since $\nu(f)=1$, the versal deformation of $f$ is trivial, i.e. has the form $f_{\lambda} = f + \lambda$, and the equivariant stability follows immediately.
\end{proof}
In particular for the nondegenerate singularity $\mu(f)=\nu(f)=1$, as $Q_f$ is equal to $\mathbb{C}$.

\begin{theorem}The existence of equivariantly simple singularities for a given group action is equivalent to the existence of equivariantly stable singularities.
\end{theorem}

\begin{proof}
In one direction this statement is self-evident as stable singularities are by definition simple.

Let $f$ be an equivariantly simple singularity. Choose $r_0 \geq |G|$ such that a neighbourhood of $r_0$-jet of $f$ intersects with only a finite number of $D^{r_0}_G$-orbits. Each of the orbits is a submerged manifold, and since a finite number of orbits cover a neighbourhood of $j^{r_0} (f)$ one of the orbits must be open. Let $g$ be the germ with an open orbit. The tangent space to the orbit of $j^{r_0} (g)$ has the maximal dimension (equal to the dimension of $\mathcal{O}^{r_0}_{G}$). Since the tangent space is the projection of the Jacobian ideal of $g$ to  $\mathcal{O}^{r_0}_{G}$ for each invariant polynomial $q$ of degree $\leq r_0$ there is an $\tilde{q} \; \in \; J_g$ with $j^{r_0} (\tilde{q}) = q$. According to Noether's theorem (see, e.g., \cite{8}, p.9) the ring of invariant polynomials can be generated by homogeneous polynomials of degree at most $|G|$,  so for each $r>r_0$ and any invariant polynomial $q$ of degree  $r$ there is also a $\tilde{q} \; \in \; J_g$ with $j^{r} (\tilde{q}) = q$, i.e. the orbit of $g$ is open for all $r>r_0$.

\end{proof}

In the course of proving theorem 1.1. in \cite{1}, we have also shown the following:

\begin{theorem} Let $f$ be the equivariantly stable singularity with respect to a linear group action $\tau$ of a group of prime order $\mathbb{Z}_p$. $\mu_{ \mathbb{Z}_p}(f)$ has one of the following isomorphism classes as a linear representation of $\mathbb{Z}_p$:

1)$\tilde{\mathbb{C}}, det(\tau)=1$;

2)$det(\tau) \otimes W, det(\tau) \neq 1$;

3)$2 \; det(\tau) \oplus (det(\tau) \otimes W), det(\tau) \neq 1$;

4)$\tilde{\mathbb{C}} \oplus 2 W, det(\tau)=1$.

Here $\tilde{\mathbb{C}}$ denotes the trivial one-dimensional representation of $\mathbb{Z}_p$ and $W$ denotes the restriction of the regular representation of $\mathbb{Z}_p$ to $\{x_1 + \ldots + x_p=0\}$.

\end{theorem}

This statement will be used in the following section.
\section{Proof of theorem 1.2.}
 \begin{theorem}
 
 There is no equivariantly stable $\mathbb{Z}_p$-invariant function germ $f$ with $\mu_{ \mathbb{Z}_p}(f)=\tilde{\mathbb{C}} \oplus 2 W$ 
 
 \end{theorem}
 
 \begin{proof} 
 
Assume that an equivariantly stable function $f$ has equivariant Milnor number equal to  $\tilde{\mathbb{C}} \oplus 2 W$. Then $det(\tau)=1$. If $det(\tau)=1$ the Hessian determinant of any invariant function is again an invariant function. It is known that  Hessian determinant of $f$ does not belong to the Jacobian ideal of $f$ for any $f$ with finite Milnor number (\cite{9}, p.100). Since there is only one invariant monomial in the monomial basis of $Q_f$ (a unit) this means that the Hessian deteminant of $f$ is nonzero at the origin, so $f$ has a nondegenerate singularity and  $\mu_{ \mathbb{Z}_p}(f)=\tilde{\mathbb{C}}$.
 
 \end{proof}

The rest of nontrivial isomorphism classes can be realized by germs $f=x^p$ (with the obvious $\mathbb{Z}_p$-action) and $x^p + x y^2$ (with the action $(x,y) \longrightarrow (\varepsilon x, \varepsilon^{\frac{p-1}{2}} y)$ for a nontrivial root of unity $\varepsilon$ of degree $p$).

\medskip

Let $f$ be a germ equivariantly stable with respect to a linear $\mathbb{Z}_p$-action $\tau$ for some prime p. 
Terms of degree 2 in the Taylor series of a stable germ must add up to a quadratic form of maximal rank  $rk(\tau)$. According to the splitting lemma in some local system of coordinates  $f$ is equal to $x_1^2 + \ldots + x_{rk(\tau)}^2 + h(x_{rk(\tau)+1}, \ldots, x_n)$, with $j^2 (h)=0$. Then $\mu(f) = \mu(h)$, $\mu(h) \geq 2^{n-rk(\tau)}$ (as $h$ has no terms of degree 2 in its Taylor series). The subspace $\{x_1= \ldots = x_{rk(\tau)} = 0\}$ is invariant (since it is the kernel of an invariant bilinear form). 

Since the Jacobian algebras of $f$ and $h$ are isomorphic $\nu(h) = \nu(f) = 1$, i.e. $h$ is also an equivariantly stable function (with respect to $\mathbb{Z}_p$-action  on  $\{x_1= \ldots = x_{rk(\tau)} = 0\}$). Due to theorem 4.1. it is now known that $\mu_{ \mathbb{Z}_p}(h)$ for $h$ with $\nu(h)=1$ can only equal $\tilde{\mathbb{C}}$ if $det(\tau)=1$ or either $det(\tau) \otimes W$ or $2 \; det(\tau) \oplus (det(\tau) \otimes W)$ if $det(\tau) \neq 1$. Therefore the Milnor number of $h$ is at most $p+1$ if $det(\tau) \neq 1$ or is equal to $1$ if $det(\tau) = 1$. Since $\mu(f) = \mu(h) \geq 2^{n-rk(\tau)}$ this concludes the proof of theorem 1.2.

\section{Proof of theorems 1.3. and 1.4.}

Let $n\geq2$ be an positive integer, $d_1, \ldots, d_n \geq 2$ -- natural numbers, $g$ -- generator of a group  $\mathbb{Z}_{d_1\ldots d_n +  (-1)^{n-1}}$,  $\epsilon$ -- primitive root of unity $1$ of degree $d_1\ldots d_n +  (-1)^{n-1}$.
 $G_n = \mathbb{Z}_{d_1\ldots d_n + (-1)^{n-1}}$ acts on $\mathbb{C}^n$ in the following fashion: 
 
 \begin{equation}{\nonumber}
g \circ (x_1, \ldots, x_n) = (\epsilon x_1, \epsilon^{-d_1} x_2, \ldots, \epsilon^{(-1)^{n-2} d_1 \ldots d_{n-2}} x_{n-1}, \epsilon^{(-1)^{n-1} d_1 \ldots d_{n-1}} x_n)
 \end{equation}

 The polynomial $f = x_1^{d_1} x_2 + x_2^{d_2} x_3 + \ldots + x_{n-1}^{d_{n_1}} x_n + x_n^{d_n} x_1$ is invariant with respect to this action. It has an isolated critical point, as it is an invertible polynomial of the loop type (see., e.g., \cite{10,11}), and its Milnor number is equal to $d_1\ldots d_n$.
 
\begin{theorem}
$f$ is an equivariantly stable $G_n$-invariant germ.
\end{theorem}

\begin{proof}

Consider the case of odd $n$, so that $G_n= \mathbb{Z}_{d_1\ldots d_n + 1}$. To prove that $f$ is stable we need to establish that $\nu(f)=1$, i.e. that the only invariant element in some monomial basis of $Q_f$ is $1$. The monomial basis of $Q_f$ consists (see, e.g. \cite{10}) of monomials $x_1^{r_1}\ldots x_n^{r_n}, 0 \leq r_i < d_i$. Proving that such monomial is only invariant when $r_1 = r_2 = \ldots = r_n =0$ is obviously equivalend to proving that 

$r_1 - r_2 d_1 + \ldots + (-1)^{n-1} r_n d_1 d_2 \ldots d_{n-1} \equiv 0 (mod \; d_1\ldots d_n + 1)$ only if $r_1 = r_2 = \ldots = r_n =0$ for $r_i < d_i$.

 $| r_1 - r_2 d_1 + \ldots + (-1)^{n-1}  r_n d_1 d_2 \ldots d_{n-1} | \leq | r_1| + |r_2 d_1| + \ldots + | r_n d_1 d_2 \ldots d_{n-1} |$. As $r_i < d_i$  $| r_1| + |r_2 d_1| + \ldots + | r_n d_1 d_2 \ldots d_{n-1} | \leq (d_1-1) + (d_2-1)d_1 + \ldots + (d_n - 1) d_1 d_2 \ldots d_{n-1}$. The last sum is a telescopic series equal to $d_1 d_2 \ldots d_n - 1$. 
 
 Therefore $r_1 - r_2 d_1 + \ldots + (-1)^{n-1}  r_n d_1 d_2 \ldots d_{n-1} \equiv 0 (mod \; d_1\ldots d_n + 1)$ 
 
 only if $r_1 - r_2 d_1 + \ldots + (-1)^{n-1}  r_n d_1 d_2 \ldots d_{n-1} = 0$. 
 
 By a similar argument, the sum of first $n-1$ terms in the equality $r_1 - r_2 d_1 + \ldots + (-1)^{n-1}  r_n d_1 d_2 \ldots d_{n-1} = 0$ has absolute value at most $d_1 d_2 \ldots d_{n-1} - 1$, so $r_1 - r_2 d_1 + \ldots + (-1)^{n-1}  r_n d_1 d_2 \ldots d_{n-1} = 0$ means $r_n=0$. Proceeding inductively, we obtain $r_1 = r_2 = \ldots = r_n =0$. 
 
 The case of even $n$ is basically identical, it only needs to be noted that the inequality  $| r_1 - r_2 d_1 + \ldots + (-1)^{n-1}  r_n d_1 d_2 \ldots d_{n-1} | \leq d_1 \ldots d_n -1$ is actually strict, as the sum $r_1 - r_2 d_1 + \ldots + (-1)^{n-1}  r_n d_1 d_2 \ldots d_{n-1}$ always has either some terms of different signs or some zero terms.
\end{proof}

If $d_1\ldots d_n + (-1)^{n-1}$ is a prime number $p$ such $f$ gives us an example of a stable (and therefore simple) $\mathbb{Z}_p$-invariant function of corank $n$. We will concentrate on odd $n$, to be definite, so $d_1 \ldots d_n = p-1$. Hence, to prove that the corank of a stable $f$ can be arbitrarily large it just needs to be shown that for each $N$ there is a prime number $p$ such that $p-1$ has a decomposition $p-1=d_1 \ldots d_n$ with at least $N+1$ nontrivial factors. This can be proven using a theorem of Erd\H{o}s (\cite{12}):
 
 \begin{theorem}
 
 As $n$ tends to infinity for any $\varepsilon>0$ and for allmost all prime $p$ not exceding $n$ the number $p-1$ has more than $(1-\varepsilon)log(log(n))$ distinct prime divisors.

 \end{theorem}

This in fact means that for ``average'' $p$ $\mathbb{Z}_p$ has an action with a simple singularity with corank at least $O(log(log(p)))$. We don't need our divisors $d_1, \ldots , d_n$ to be distinct or prime, however the function $\Omega(n)$ counting prime divisors with multiplicity has the same asymptotic behaviour as the number of distinct prime divisors (\cite{13}, sections 22.10 and 22.11), so it can be reasonably assumed that counting the total number of not necessary distinct divisors $d_1, \ldots, d_n$ of $p-1$ will not lead to a different estimate.

We will now prove that the corank of a stable invariant singularity in fact has logarithmic growth.

\begin{theorem}There are constants $\alpha, \beta$ and an infinite sequence of distinct prime numbers $p_1, \ldots p_n$ such that $p_i - 1$ has at least $\alpha log_2 (p_i - 1) + \beta$ prime divisors, counted with multiplicity.

\end{theorem}

\begin{proof}

By Linnik's theorem (\cite{14}) there are absolute constants $C, L$ such that every arithmetical progression $a + kd, k \; \in \; \mathbb{Z}_{\geq 0}$ with coprime $a, d$ and $0<a<d$ contains at least one prime $p < C d^L$. Consider $a=1, d= 2^m$ and pick the first prime in the arithmetic progression $a + kd$. $p-1$ is divisible by $2^m$, so $\Omega(p-1) \geq m$. As $p < C d^L = C 2^{mL}$, $log_2 (p-1) < log_2(C) + mL$. Therefore $\Omega(p-1) >\frac{ log_2 (p-1) - log_2 (C)}{L}$.  By choosing each succesive $2^m$ to be strictly greater than all the previously obtained $p_i$ we can build an infinite sequence of such $p$.

\end{proof}

It should be noted that the constant $L$ in Linnik's theorem was later estimated to be at most $5$ (\cite{15}).

\end{document}